\documentclass[11pt,a4paper]{amsart}
\usepackage{amssymb}
\usepackage[all]{xy}
\usepackage{graphicx}

\newtheorem{theorem}{Theorem}[section]
\newtheorem{lem}[theorem]{Lemma}
\newtheorem{thm}[theorem]{Theorem}

\newtheorem{cor}[theorem]{Corollary}

\newtheorem{ex}[theorem]{Example}
\newtheorem{defn}[theorem]{Definition}
\newtheorem{rem}[theorem]{Remark}

\newcommand{\Hom}{{\rm Hom}}
\newcommand{\End}{{\rm End}}

\begin{document}

\title[Galois connections]{Essential and retractable Galois connections}

\author[S. Crivei]{Septimiu Crivei}

\address{Faculty of Mathematics and Computer Science,``Babe\c s-Bolyai" University, Str. Mihail Kog\u alni-ceanu 1,
400084 Cluj-Napoca, Romania} \email{crivei@math.ubbcluj.ro}


\begin{abstract} For bounded lattices, we introduce certain Galois connections, called (cyclically) essential,
retractable and UC Galois connections, which behave well with respect to concepts of module-theoretic nature
involving essentiality. We show that  essential retractable Galois connections preserve uniform dimension, whereas
essential retractable UC Galois connections induce a bijective correspondence between sets of closed
elements. Our results are applied to suitable Galois connections between submodule lattices. Cyclically essential Galois
connections unify semi-projective and semi-injective modules, while retractable Galois connections unify retractable and
coretractable modules. 
\end{abstract}

\thanks{The author acknowledges the support of the grant PN-II-RU-TE-2011-3-0065.}

\maketitle

\section{Introduction}

If $M$ and $N$ are two right $R$-modules, then it is well-known that $\Hom_R(M,N)$ has a structure of $({\rm
End}_R(N),{\rm End}_R(M))$-bimodule. An interesting classical problem in module theory is to relate properties of the
modules $M$ and $N$ with properties of the bimodule $\Hom_R(M,N)$, and in particular, to relate properties of a module
with properties of its endomorphism ring. The present paper is motivated by this general problem, and its goal is to
show how one can efficiently use Galois connections in order to obtain and clarify the desired results by using easier
to handle and more transparent conditions. 

We introduce certain Galois connections, called (cyclically) essential, retractable and UC
Galois connections, between two bounded lattices. Essential retractable Galois connections turn out to
preserve uniform dimension. Also, we prove that essential retractable UC Galois connections between
two bounded modular lattices $A$ and $B$ induce a bijective correspondence between the set of closed elements of $A$ and
the set of closed elements of $B$. As an application, we show how the extending property transfers through
essential retractable UC Galois connections. Our paper expands the approach from \cite{CIKO}, not only dualizing, but
also generalizing results involving essentiality in bounded (modular) lattices. They can be applied to some
particular Galois connections between submodule lattices, previously pointed out by Albu and N\u ast\u asescu
\cite[pp.~25-26]{AN}. Then retractable Galois connections particularize to retractability and
coretractability of modules, cyclically essential Galois connections particularize to semi-projectivity and
semi-injectivity of modules, and our results immediately yield and unify recent properties from \cite{AES,HV}.
Thus it becomes clear that both retractability and coretractability are instances of the same notion, and
this is also the case for semi-projectivity and semi-injectivity. When $M$ is a right $R$-module and $N$ is an
$M$-faithful right $R$-module, we deduce a module-theoretic result of Zelmanowitz on the existence of a
bijective correspondence between the sets of closed right $R$-submodules of $M$ and closed right ${\rm
End}_R(M)$-submodules of $\Hom_R(M,N)$ \cite[Theorem~1.2]{Z}.

\section{Special Galois connections}

Let us recall the concept of (monotone) Galois connection (e.g., see \cite{Erne}). 

\begin{defn} \rm Let $(A,\leq)$ and $(B,\leq)$ be lattices. A {\it Galois connection} between them consists of a
pair $(\alpha, \beta)$ of two order-preserving functions $\alpha:A\to B$ and $\beta:B\to A$ such that for all $a\in
A$ and $b\in B$, we have $\alpha(a)\leq b$ if and only if $a\leq \beta(b)$. Equivalently, $(\alpha,\beta)$ is a Galois
connection if and only if for all $a\in A$, $a\leq \beta\alpha(a)$ and for all $b\in B$, $\alpha\beta(b)\leq b$. 

An element $a\in A$ (respectively $b\in B$) is called \emph{Galois} if $\beta\alpha(a)=a$ (respectively
$\alpha\beta(b)=b$). 
\end{defn}

Note that a Galois connection between two posets $(A,\leq)$ and $(B,\leq)$ is nothing else but an adjoint pair between
the categories $A$ and $B$, whose objects are their elements, and whose morphisms are pairs $(x,y)$ with $x\leq y$. 

Throughout the paper, we view any lattice $A$ both with a subjacent poset structure $(A,\leq)$ and as a triple
$(A,\wedge,\vee)$, where $\wedge$ and $\vee$ denote the infimum and the supremum of elements in $A$. Recall that $A$ is
bounded if it has a least element, denoted by $0$, and a greatest element, denoted by $1$. If $A$ is bounded, then we
view it as $(A,\wedge,\vee,0,1)$, and we tacitly assume that $0\neq 1$. For $a,a'\in A$, we also denote $[a,a']=\{x\in
A\mid a\leq x\leq a'\}$.

The following well-known results on Galois connections (e.g., see \cite[Proposition~3.3]{AN}, \cite{Erne}) will be
freely used throughout the paper. 

\begin{lem} \label{l:galois} Let $(\alpha,\beta)$ be a Galois connection between two lattices $A$ and $B$. Then: 

(i) $\alpha\beta\alpha=\alpha$ and $\beta\alpha\beta=\beta$.

(ii) $\alpha$ preserves all suprema in $A$ and $\beta$ preserves all infima in $B$. 

(iii) If $A$ and $B$ are bounded, then $\alpha(0)=0$ and $\beta(1)=1$. 

(iv) The restrictions of $\alpha$ and $\beta$ to the corresponding sets of Galois elements are mutually inverse
bijections.
\end{lem}

One may consider the following notions of module-theoretic nature (see \cite{DHSW, GP, Schmidt}).

\begin{defn} \label{d:ess} \rm Let $A$ be a bounded lattice. 

An element $a\in A$ is called:

(1) {\it essential} in $A$ if $a\wedge x=0$ implies $x=0$ for $x\in A$.

(2) {\it closed} in $A$ if $a$ essential in $[0,a']$ implies $a=a'$ for $a'\in A$.

(3) {\it closure} of $a\in A$ in $A$ if $a$ is essential in $[0,a']$ and $a'$ is closed in $A$.

(4) {\it cyclic} if $[0,a]$ is a distributive lattice and satisfies the ascending chain condition.

The lattice $A$ is called:

(5) {\it UC} if every $a\in A$ has a unique closure in $A$, which will be denoted by $\bar
a$.

(6) {\it uniform} if every non-zero element of $A$ is essential in $[0,1]$. 

(7) {\it cyclically generated} if every element of $A$ is a join of cyclic elements.
\end{defn}

The following lemma has a similar proof as for modules (see \cite{DHSW, GP}). 

\begin{lem} \label{l:ess} Let $A$ be a bounded lattice.

(i) Let $a,b,c\in A$ be such that $a\leq b\leq c$. Then $a$ is essential in $[0,c]$ if and only if $a$ is
essential in $[0,b]$ and $b$ is essential in $[0,c]$. If $a$ is essential in $[0,c]$, then $a\wedge b$ is essential in
$[0,b]$.

(ii) If $A$ is cyclically generated, then an element $a\in A$ is essential in $A$ if and only if $a\wedge x=0$ implies
$x=0$ for cyclic elements $x\in A$.

(iii) If $A$ is modular, then every complement in $A$ is closed in $A$.

(iv) If $A$ is modular, then every element of $A$ has a closure in $A$. 

(v) If $A$ is modular UC and $a_1,a_2\in A$ are such that $a_1\leq a_2$, then their closures satisfy
$\overline{a_1}\leq \overline{a_2}$.  
\end{lem}

We introduce some special types of Galois connections.

\begin{defn} \rm A Galois connection $(\alpha,\beta)$ between two bounded lattices $A$ and $B$ is called: 

(1) {\it (cyclically) essential} if for every (cyclic) $a\in A$, $a$ is essential in $[0,\beta\alpha(a)]$.

(2) {\it retractable} if for every $b \in B$, $\alpha\beta(b)$ is essential in $[0,b]$.

(3) {\it UC} if for every closed element $b\in B$, $b$ is the unique closure of $\alpha\beta(b)$ in $B$. 
\end{defn}

\begin{lem} \label{l:clgal} Let $(\alpha,\beta)$ be a Galois connection between two bounded lattices $A$ and $B$.

(i) If $(\alpha,\beta)$ is essential, then every closed element of $A$ is Galois and $\beta(0)=0$. 

(ii) $(\alpha,\beta)$ is retractable if and only if $\beta(b)=0$ implies $b=0$ for $b\in B$. 

(iii) If $(\alpha,\beta)$ is (cyclically) essential retractable, and $a,a'\in A$ are (cyclic) such that $a\wedge a'=0$,
then $\alpha(a)\wedge \alpha(a')=0$. 

(iv) If $(\alpha,\beta)$ is essential retractable, and $\alpha(1)=1$, then $\alpha$ preserves complements.
\end{lem}

\begin{proof} (i) This is clear. 

(ii) Assume first that $(\alpha,\beta)$ is retractable. Let $b\in B$ be such that $\beta(b)=0$. Then
$\alpha\beta(b)=\alpha(0)=0$, and so $\alpha\beta(b)\wedge b=0$. Since $\alpha\beta(b)$ is essential in $[0,b]$,
it follows that $b=0$.

Now assume that $\beta(b)=0$ implies $b=0$ for $b\in B$. Let $b,b'\in B$ be such that $b'\leq b$ and
$\alpha\beta(b)\wedge b'=0$. Then we have $0=\beta(0)=\beta\alpha\beta(b)\wedge \beta(b')=\beta(b)\wedge
\beta(b')=\beta(b\wedge b')=\beta(b')$, which implies $b'=0$. This shows that $\alpha\beta(b)$ is essential in
$[0,b]$, and so $(\alpha,\beta)$ is retractable.

(iii) Since $(\alpha,\beta)$ is (cyclically) essential, $a\wedge a'=0$ implies $\beta\alpha(a)\wedge \beta\alpha(a')=0$,
and so, $\beta(\alpha(a)\wedge \alpha(a'))=0$. Since $(\alpha,\beta)$ is retractable, one must have $\alpha(a)\wedge
\alpha(a')=0$ by (ii).  

(iv) This is clear by (iii).
\end{proof}

We end this section with some results on the transfer of the essential property through retractable (essential)
Galois connections.

\begin{lem} \label{l:transfer} Let $(\alpha,\beta)$ be a retractable Galois connection between two bounded lattices $A$
and $B$ such that $\beta(0)=0$. 

(i) Let $b,b'\in B$ be such that $\beta(b)$ is essential in $[0,\beta(b')]$. Then $b\wedge b'$ is essential in
$[0,b']$.

(ii) Let $a,a'\in A$ be such that $a$ is essential in $[0,a']$ and $a'$ is a Galois element. Then $\alpha(a)$ is
essential in $[0,\alpha(a')]$. 

Assume in addition that $(\alpha,\beta)$ is (cyclically) essential (and $A$ is cyclically generated).

(iii) Let $b,b'\in B$ be such that $b$ is essential in $[0,b']$. Then $\beta(b)$ is essential in $[0,\beta(b')]$.

(iv) Let $a,a'\in A$ (with $a$ cyclic) be such that $\alpha(a)$ is essential in $[0,\alpha(a')]$. Then $a\wedge a'$ is
essential in $[0,a']$. 
\end{lem}

\begin{proof} (i) Let $b''\in B$ be such that $b''\leq b'$ and $(b\wedge b')\wedge b''=0$. Then $b\wedge b''=0$. Since
$\beta(0)=0$, we have $\beta(b)\wedge \beta(b'')=0$. Since $\beta(b)$ is essential in $[0,\beta(b')]$ and
$\beta(b'')\leq \beta(b')$, it follows that $\beta(b'')=0$. Since $(\alpha,\beta)$ is retractable, we have $b''=0$.
This shows that $b\wedge b'$ is essential in $[0,b']$. 

(ii) We have $a\leq \beta\alpha(a)\leq \beta\alpha(a')=a'$. Since $a$ is essential in $[0,a']$, Lemma \ref{l:ess}
implies that $\beta\alpha(a)$ is essential in $[0,\beta\alpha(a')]$. Then $\alpha(a)$ is essential in $[0,\alpha(a')]$
by (i).

(iii) Let $a\in A$ be (cyclic) such that $a\leq \beta(b')$ and $\beta(b)\wedge a=0$. Since $(\alpha,\beta)$ is
(cyclically) essential, it folows that $\beta(b)\wedge \beta\alpha(a)=0$, and so $\beta(b\wedge \alpha(a))=0$. Since
$(\alpha,\beta)$ is retractable, $b\wedge \alpha(a)=0$. Then $b'\wedge \alpha(a)=0$, because $b$ is essential in
$[0,b']$. Since $\alpha(a)\leq b'$, it follows that $\alpha(a)=0$. Then $a\leq \beta\alpha(a)=0$, and
so $a=0$. This shows that $\beta(b)$ is essential in $[0,\beta(b')]$.

(iv) Since $(\alpha,\beta)$ is (cyclically) essential, $a$ is essential in $[0,\beta\alpha(a)]$. Also, $\alpha(a)$
essential in $[0,\alpha(a')]$ and (iii) imply that $\beta\alpha(a)$ is essential in $[0,\beta\alpha(a')]$. By Lemma
\ref{l:ess}, it follows that $a$ is essential in $[0,\beta\alpha(a')]$, and furthermore, $a\wedge a'$ is essential in
$[0,a']$.
\end{proof}

\section{Examples}

Let us see some relevant examples illustrating the above theory. For properties of essential, closed or unique closure
subgroups of abelian groups the reader is referred to \cite{CO,CS}.

\begin{ex} \rm \label{e:ex} (1) Consider the abelian group $G=\mathbb{Z}_{p^2}\times \mathbb{Z}_{q^2}$ for some primes
$p$ and $q$ with $p\neq q$, where $\mathbb{Z}_n$ denotes the cyclic group of
order $n\in \mathbb{N}$. The subgroup lattice $L(G)$ of $G$ is given by the left hand side diagram: 
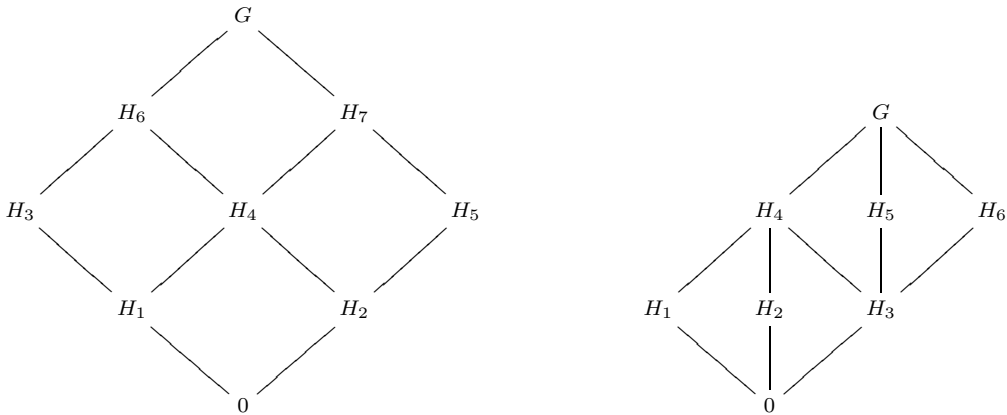
\begin{scriptsize} 
\[\SelectTips{cm}{}
\xymatrix{
 && G \ar@{-}[dl] \ar@{-}[dr] & & &&&&& \\ 
& H_6 \ar@{-}[dl] \ar@{-}[dr] & & H_7 \ar@{-}[dl] \ar@{-}[dr] & & &&& G \ar@{-}[dl] \ar@{-}[d] \ar@{-}[dr] &\\
H_3 \ar@{-}[dr] & & H_4 \ar@{-}[dl] \ar@{-}[dr] & & H_5 \ar@{-}[dl] &&& H_4 \ar@{-}[dl] \ar@{-}[d] \ar@{-}[dr] & H_5
\ar@{-}[d] & H_6 \ar@{-}[dl]\\
& H_1 \ar@{-}[dr] & & H_2 \ar@{-}[dl] &&& H_1 \ar@{-}[dr] & H_2 \ar@{-}[d] & H_3 \ar@{-}[dl] \\
 & & 0 & & &&& 0 & &   
}\]
\end{scriptsize} 

Consider the functions $\alpha:L(G)\to L(G)$ defined by $\alpha(0)=0$, $\alpha(H_1)=H_1$, $\alpha(H_2)=\alpha(H_5)=H_2$,
$\alpha(H_3)=H_1$, $\alpha(H_4)=\alpha(H_6)=\alpha(H_7)=\alpha(G)=H_4$, $\beta:L(G)\to L(G)$ defined by $\beta(0)=0$,
$\beta(H_1)=\beta(H_3)=H_3$, $\beta(H_2)=\beta(H_5)=H_5$, $\beta(H_4)=\beta(H_6)=\beta(H_7)=\beta(G)=G$. Then
$(\alpha,\beta)$ is a Galois connection from the lattice $(L(G),\subseteq)$ to itself. For every $H\in
\{H_0,H_3,H_5,G\}$ we have $\beta\alpha(H)=H$. Hence $0$, $H_3$, $H_5$ and $G$ are Galois elements in $A=L(G)$.
Also, $H_1$ is essential in $[0,\beta\alpha(H_1)]=[0,H_3]$, $H_2$ is essential in $[0,\beta\alpha(H_2)]=[0,H_5]$, $H_2$
is essential in $[0,\beta\alpha(H_4)]=[0,G]$, $H_6$ is essential in $[0,\beta\alpha(H_6)]=[0,G]$, and $H_7$ is essential
in $[0,\beta\alpha(H_7)]=[0,G]$. On the other hand, for every $H\in \{0,H_1,H_2,H_4\}$ we have
$\alpha\beta(H)=H$. Hence $0$, $H_1$, $H_2$ and $H_4$ are Galois elements in $B=L(G)$. Also, $\alpha\beta(H_3)=H_1$ is
essential in $[0,H_3]$, $\alpha\beta(H_5)=H_2$ is essential in $[0,H_5]$, $\alpha\beta(H_6)=H_4$ 
is essential in $[0,H_6]$, $\alpha\beta(H_7)=H_4$ is essential in $[0,H_7]$ and $\alpha\beta(G)=H_4$ is
essential in $[0,G]$. Moreover, for every closed $H\in L(G)$, that is, $H\in \{0,H_3,H_5,G\}$, $H$ is the unique closure
of $\alpha\beta(H)$ in $L(G)$. Hence $(\alpha,\beta)$ is an essential retractable UC Galois connection. Note that $H_3$
is a closed element, but not a Galois element in $B=L(G)$. Hence not every closed element of $B$ is Galois.

(2)  Consider the abelian group $G=\mathbb{Z}_2\times \mathbb{Z}_4$. The subgroup lattice $L(G)$ of $G$ is given by the
right hand side of the above diagram.

Consider the functions $\alpha:L(G)\to L(G)$ defined by $\alpha(H_6)=H_3$, $\alpha(G)=H_4$ and
$\alpha(H)=H$ for every $H\in L(G)\setminus \{H_6,G\}$, and $\beta:L(G)\to L(G)$ defined by
$\beta(H_3)=H_6$, $\beta(H_4)=G$, and $\beta(H)=H$ for every $H\in L(G)\setminus \{H_3,H_4\}$. Then
$(\alpha,\beta)$ is a Galois connection from the lattice $(L(G),\subseteq)$ to itself. For every $H\in
L(G)\setminus \{H_3,H_4\}$ we have $\beta\alpha(H)=H$. Also, $H_3$ is essential in
$[0,\beta\alpha(H_3)]=[0,H_6]$ and $H_4$ is essential in $[0,\beta\alpha(H_4)]=[0,G]$. Hence
$(\alpha,\beta)$ is essential retractable. But $(\alpha,\beta)$ is not UC, because for the closed subgroup $H_6$,
$\alpha\beta(H_6)=H_3$ has two closures in $G$, namely $H_5$ and $H_6$. 

(3) In general, if $A$ is uniform, $B$ is uniform or $B$ is UC, then $(\alpha,\beta)$ is clearly essential, retractable
or UC respectively. But none of the converse implications holds. For the first two, a counterexample is given by (1),
and for the last one, take $A=B=L(G)$ with $G=\mathbb{Z}_2\times \mathbb{Z}_4$, and $\alpha$, $\beta$ the identity
maps. 
\end{ex}

\begin{ex} \rm (1) Consider the following lattices, denoted by $A$, $B$ and $C$ respectively:
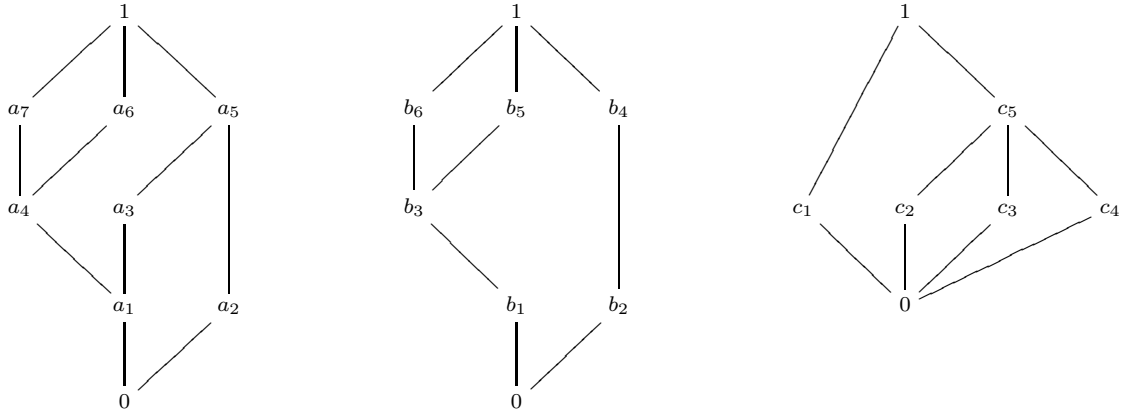
\begin{scriptsize} 
\[\SelectTips{cm}{}
\xymatrix{
 & 1 \ar@{-}[dl] \ar@{-}[d] \ar@{-}[dr] & &&& 1 \ar@{-}[dl] \ar@{-}[d] \ar@{-}[dr] & &&& 1 \ar@{-}[ddl] \ar@{-}[dr] &&
\\ 
a_7 \ar@{-}[d] & a_6 \ar@{-}[dl] & a_5 \ar@{-}[dl] \ar@{-}[dd] && b_6 \ar@{-}[d] & b_5 \ar@{-}[dl] & b_4 \ar@{-}[dd]
&&&& c_5 \ar@{-}[dl] \ar@{-}[d] \ar@{-}[dr] & \\
a_4 \ar@{-}[dr] & a_3 \ar@{-}[d] & & & b_3 \ar@{-}[dr] & & && c_1 \ar@{-}[dr] & c_2 \ar@{-}[d] & c_3 \ar@{-}[dl] & c_4
\ar@{-}[dll] \\
 & a_1 \ar@{-}[d] & a_2 \ar@{-}[dl] & & & b_1 \ar@{-}[d] & b_2 \ar@{-}[dl] && & 0 && \\
 & 0 & & & & 0 & &&&&& \\
}\]
\end{scriptsize} 

Consider the functions $\alpha:A\to B$ defined by $\alpha(0)=0$, $\alpha(a_1)=b_1$, $\alpha(a_2)=b_4$,
$\alpha(a_3)=1$, $\alpha(a_4)=b_1$, $\alpha(a_5)=1$, $\alpha(a_6)=1$, $\alpha(a_7)=b_3$ and $\alpha(1)=1$,
$\beta:B\to A$ defined by $\beta(0)=0$, $\beta(b_1)=a_4$, $\beta(b_2)=0$, $\beta(b_3)=a_7$, $\beta(b_4)=a_2$,
$\beta(b_5)=a_7$, $\beta(b_6)=a_7$ and $\beta(1)=1$. Then $(\alpha,\beta)$ is a Galois connection between $A$ and $B$.
But $(\alpha,\beta)$ is neither essential, nor retractable, nor UC. Indeed, for instance, $a_3$ is not essential in
$[0,\beta\alpha(a_3)]=[0,1]$, $0=\alpha\beta(b_2)$ is not essential in $b_2$, and $\alpha\beta(b_5)=b_3$ has closures
$b_5$ and $b_6$. 

Note that $a_4=\beta\alpha(a_4)$ is a Galois element, but not a closed element in $A$, and $b_3=\alpha\beta(b_3)$ is a
Galois element, but not a closed element in $B$. Also, $\alpha$ and $\beta$ do not preserve essentiality. For
instance, $a_4$ is essential in $[0,a_6]$, but $\alpha(a_4)=b_1$ is not essential in $[0,\alpha(a_6)]=[0,1]$; $b_2$ is
essential in $[0,b_4]$, but $\beta(b_2)=0$ is not essential in $[0,\beta(b_4)]=[0,a_2]$. 

(2) Consider the same functions $\alpha,\beta$ as in (1) except for the values $\alpha(a_2)=b_2$ and
$\beta(b_2)=a_2$. Then $(\alpha,\beta)$ is a Galois connection between $A$ and $B$, which is still neither essential
nor UC. But it is retractable by Lemma \ref{l:clgal}.

(3) Consider the above lattice $C$ and note that its cyclic elements are $0,c_1,c_2,c_3$ and $c_4$. Consider the
functions $\alpha,\beta:C\to C$ defined by $\alpha(0)=0$, $\alpha(c_i)=c_i$ for every $i\in \{1,2,3,4\}$,
$\alpha(c_5)=\alpha(1)=c_5$, and $\beta(0)=0$, $\beta(c_i)=c_i$ for every $i\in \{1,2,3,4\}$, $\beta(c_5)=\beta(1)=1$.
Then $(\alpha,\beta)$ is a cyclically essential Galois connection between $C$ and itself. But $(\alpha,\beta)$ is not
essential, because $c_5$ is not essential in $[0,\beta\alpha(c_5)]=[0,1]$. 
\end{ex}

Now we recall some relevant Galois connections between submodule lattices, previously pointed out in the
literature (e.g., see \cite{AN}). Throughout the rest of the paper we shall freely use the following notation. 

Let $R$ be an associative ring with (non-zero) identity. Let $M$ and $N$ be two right $R$-modules, and denote
$U=\Hom_R(M,N)$, $S=\End_R(M)$ and $T=\End_R(N)$. Then $_TU_S$ is a bimodule. 

For every submodule $X$ of $M_R$ and every submodule $Z$ of ${}_TU$, we denote: 
\[l_U(X)=\{f\in U\mid X\subseteq {\rm Ker}(f)\}, \quad r_M(Z)=\bigcap_{f\in Z}{\rm Ker}(f).\]

For every submodule $Y$ of $N_R$ and every submodule $Z$ of $U_S$, we denote: 
\[l'_U(Y)=\{f\in U\mid {\rm Im}(f)\subseteq Y\}, \quad r'_N(Z)=\sum_{f\in Z}{\rm Im}(f).\]

\begin{thm} {\rm \cite[Proposition~3.4]{AN}} $(r_M,l_U)$ is a Galois connection between the submodule lattices
$L({}_TU)$ and $L(M_R)^{\rm op}$, and $(r'_N,l'_U)$ is a Galois connection between $L(U_S)$ and $L(N_R)$.  
\end{thm}

Let $\sigma[M]$ be the full subcategory of the category of right $R$-modules consisting of all submodules of
$M$-generated modules. 

A module $N\in \sigma[M]$ is called \emph{$M$-retractable} if $\Hom_R(M,D)\neq 0$ for every non-zero submodule $D$ of
$N_R$. A module $M\in \sigma[N]$ is called \emph{$N$-coretractable} if $\Hom(M/C,N)\neq 0$ for every proper submodule
$C$ of $M_R$. 

A module $M\in \sigma[N]$ is called \emph{$N$-semi-projective} if for every submodule $D$ of $N_R$, every epimorphism
$g:N\to D$ and every homomorphism $\beta:M\to D$, there exists a homomorphism $\gamma:M\to N$ such that
$g\gamma=\beta$. A module $N\in \sigma[M]$ is called \emph{$M$-semi-injective} if for every submodule $C$ of $M_R$,
every monomorphism $f:M/C\to M$ and every homomorphism $\alpha:M/C\to N$, there exists a homomorphism $\gamma:M\to N$
such that $\gamma f=\alpha$. Taking $N=M$ one obtains the notions of semi-projectivity and semi-injectivity (e.g., see
\cite[4.20]{CLVW} and \cite[p. 261]{Wis}).

\begin{thm} \label{t:retcoret} (i) $N$ is $M$-retractable if and only if $(r'_N,l'_U)$ is retractable.

(ii) $M$ is $N$-coretractable if and only if $(r_M,l_U)$ is retractable.

(iii) If $M$ is $N$-semi-projective, then $(r'_N,l'_U)$ is cyclically essential.  

(iv) If $N$ is $M$-semi-injective, then $(r_M,l_U)$ is cyclically essential.  

(v) If $N$ is $M$-faithful, then $(r_M,l_U)$ is essential retractable UC. 
\end{thm}

\begin{proof} (i), (ii) These follow by Lemma \ref{l:clgal}.

(iii) If $M$ is $N$-semi-projective, then $\Hom_R(M,ZM)=Z$ for every cyclic submodule $Z$ of $U_S$, hence
$l'_Ur'_N(Z)=Z$ for every cyclic submodule $Z$ of $U_S$.

(iv) If $N$ is $M$-semi-injective, then $l_Ur_M(f)=Tf$ for every $f\in {}_TU$, hence $l_Ur_M(Z)=Z$ for every cyclic
submodule $Z$ of ${}_TU$.

(v) This follows by \cite[Proposition~1.1]{Z}.
\end{proof}

\section{Uniform dimension}

Let $X$ be a bounded modular lattice. Recall that a subset $Y$ of $X\setminus \{0\}$ is called
\emph{join-independent} if $(y_1\vee \ldots \vee y_n)\wedge x=0$ for every finite subset $\{y_1,\dots,y_n\}$ of $Y$
and every $x\in Y\setminus \{y_1,\dots,y_n\}$. If there is a finite supremum $d$ of all numbers $k$ such that $X$
has a join-independent subset with $k$ elements, then $X$ \emph{has uniform dimension} (or \emph{Goldie dimension}) $d$;
otherwise $X$ has infinite uniform dimension (see \cite[Theorem~5]{GP}). We denote the uniform dimension of $X$
by ${\rm udim}(X)$. The dual notion is that of \emph{hollow dimension} (or \emph{dual Goldie dimension}) of $X$, which 
is denoted by ${\rm hdim}(X)$.

Let $(\alpha,\beta)$ be a Galois connection between two bounded modular lattices $A$ and $B$. In general, $\beta$ does
not preserve suprema. But the following weaker property turns out to be useful. Inspired by the behaviour of additive
functors in additive categories, we say that $\beta:B\to A$ is \emph{additive} if $\beta(b\vee b')=\beta(b)\vee
\beta(b')$ for every $b,b'\in B$ with $b\wedge b'=0$. Note that if $\beta$ is additive, then $\beta$ preserves
complements.

\begin{thm} \label{t:udim} Let $(\alpha,\beta)$ be a retractable Galois connection between two bounded modular lattices
$A$ and $B$ such that $\beta(0)=0$. Then:

(i) ${\rm udim}(B)\leq {\rm udim}(A)$. 

(ii) If $(\alpha,\beta)$ is essential, then ${\rm udim}(A)={\rm udim}(B)$. 

(iii) If $A$ is cyclically generated, $(\alpha,\beta)$ is cyclically essential and $\beta$ is additive, then ${\rm
udim}(A)={\rm udim}(B)$. 
\end{thm}

\begin{proof} (i) We prove that if $B$ has a finite join-independent subset with $m$ elements, then so has $A$.
This will imply that ${\rm udim}(B)\leq {\rm udim}(A)$ in both finite and infinite cases for ${\rm udim}(B)$. To this
end, let $\{b_1,\dots,b_m\}$ be a join-independent subset of $B$. We claim that $X=\{\beta(b_1),\dots,\beta(b_m)\}$ is
a join-independent subset of $A$ with $m$ elements. 

Suppose that $\beta(b_i)=\beta(b_j)$ for some $i,j\in \{1,\dots,m\}$ with $i\neq j$. Since $b_i\wedge b_j=0$,
we have $\beta(b_i)\wedge \beta(b_j)=\beta(b_i\wedge b_j)=\beta(0)=0$. We conclude that $\beta(b_i)=\beta(b_j)=0$,
which implies $b_i=0$, because $(\alpha,\beta)$ is retractable. But this is a contradiction. Hence $X$ has $m$
elements.

The join-independence of $\{b_1,\dots,b_m\}$ is known to be equivalent to the condition: $$(b_1\vee \ldots
\vee b_{k-1})\wedge b_k=0, \textrm{ for every } k\in \{2,\dots,m\}.$$ Let $k\in \{2,\dots,m\}$. Then
$\beta(b_1\vee \ldots \vee b_{k-1})\wedge \beta(b_k)=0$. Since $\beta(b_1\vee \ldots \vee b_{k-1})\geq \beta(b_1)\vee
\ldots \vee \beta(b_{k-1})$, it follows that $(\beta(b_1)\vee \ldots \vee \beta(b_{k-1}))\wedge \beta(b_k)=0$. This
shows that $X$ is join-independent.

(ii) Assume that $(\alpha,\beta)$ is also essential. The idea of the proof is the same as for (i), but the technical
part is slightly different and relies heavily on Lemma \ref{l:clgal}. 

We prove that if $A$ has a finite join-independent subset with $n$ elements, then so has $B$.
To this end, let $\{a_1,\dots,a_n\}$ be a join-independent subset of $A$. 

We claim that $Y=\{\alpha(a_1),\dots,\alpha(a_n)\}$ is a join-independent subset of $B$ with $n$ elements. 

Suppose that $\alpha(a_i)=\alpha(a_j)$ for some $i,j\in \{1,\dots,n\}$ with $i\neq j$. Since $a_i\wedge a_j=0$, we have
$\alpha(a_i)\wedge \alpha(a_j)=0$ by Lemma \ref{l:clgal}. We conclude that $\alpha(a_i)=\alpha(a_j)=0$,
which implies $a_i\leq \beta\alpha(a_i)=\beta(0)=0$, because $(\alpha,\beta)$ is retractable. But this is a
contradiction. Hence $Y$ has $n$ elements.

The join-independence of $\{a_1,\dots,a_n\}$ is equivalent to the condition $(a_1\vee \ldots
\vee a_{k-1})\wedge a_k=0$ for every $k\in \{2,\dots,n\}$. Let $k\in \{2,\dots,n\}$. Then $\alpha(a_1\vee \ldots
\vee a_{k-1})\wedge \alpha(a_k)=0$ by Lemma \ref{l:clgal}. Hence $(\alpha(a_1)\vee \ldots \vee
\alpha(a_{k-1}))\wedge \alpha(a_k)=0$. This shows that $Y$ is join-independent.

(iii) Assume that $A$ is cyclically generated, $(\alpha,\beta)$ is cyclically essential and $\beta$ is additive. We need
to slightly modify the proof of (ii). 

Let $\{a_1,\dots,a_n\}$ be a join-independent subset of $A$. Since $A$ is cyclically generated, one may find a
join-independent subset $\{a_1',\dots,a_n'\}$ of $A$ consisting of cyclic elements such that $a_i'\leq a_i$ for every
$i=1,\dots,n$, so we may assume that $\{a_1,\dots,a_n\}$ consists of cyclic elements. As above,
$Y=\{\alpha(a_1),\dots,\alpha(a_n)\}$ has $n$ elements. 

By Lemma \ref{l:clgal}, $\beta\alpha(a_1)\wedge \beta\alpha(a_2)=\beta(\alpha(a_1)\wedge \alpha(a_2))=\beta(0)=0$. Since
$a_1$ is essential in $[0,\beta\alpha(a_1)]$ and $a_2$ is essential in $[0,\beta\alpha(a_2)]$, it follows that $a_1\vee
a_2$ is essential in $[0,\beta\alpha(a_1)\vee \beta\alpha(a_2)]$ by \cite[Lemma~3]{GP}. Since $\beta$ is additive,
$\beta\alpha(a_1)\vee \beta\alpha(a_2)=\beta\alpha(a_1\vee a_2)$ by Lemma \ref{l:clgal}. Now the equality $(a_1\vee
a_2)\wedge a_3=0$ and the essentiality properties imply $\beta\alpha(a_1\vee a_2)\wedge \beta\alpha(a_3)=0$, that is,
$\beta(\alpha(a_1\vee a_2)\wedge \alpha(a_3))=0$. Since $(\alpha,\beta)$ is retractable, we have $\alpha(a_1\vee
a_2)\wedge \alpha(a_3)=0$, and so $(\alpha(a_1)\vee \alpha(a_2))\wedge \alpha(a_3)=0$. By induction, we obtain  
$(\alpha(a_1)\vee \ldots \vee \alpha(a_{k-1}))\wedge \alpha(a_k)=0$ for every $k\in \{2,\dots,n\}$. This shows that $Y$
is join-independent.
\end{proof}

For the following module-theoretic corollary, note that submodule lattices are cyclically generated. Concerning the
Galois connections $(r'_N,l'_U)$ and $(r_M,l_U)$, $l_U'$ is clearly additive, while $l_U$ is additive by
\cite[Lemma~4.9]{AES}. Now Theorems \ref{t:retcoret} and \ref{t:udim} immediately yield the following corollary. 

\begin{cor} \label{c:uhdim} (i) If $N$ is $M$-retractable, then ${\rm udim}(N_R)\leq {\rm
udim}(\Hom_R(M,N)_S)$. 

(ii) If $M$ is $N$-coretractable, then ${\rm hdim}(M_R)\leq {\rm udim}({}_T\Hom_R(M,N))$.

(iii) If $N$ is $M$-retractable and $M$ is $N$-semi-projective, then $${\rm udim}(N_R)={\rm
udim}(\Hom_R(M,N)_S).$$

(iv) If $M$ is $N$-coretractable and $N$ is $M$-semi-injective, then $${\rm hdim}(M_R)={\rm udim}({}_T\Hom_R(M,N)).$$

(v) If $N$ is $M$-faithful, then ${\rm udim}(N_R)={\rm udim}(\Hom_R(M,N)_S)$.
\end{cor}

Taking $N=M$ in Corollary \ref{c:uhdim}, (i)-(iv) yield \cite[Propositions~4.10 and 4.11]{AES} and
\cite[Theorem~2.6]{HV}, while (v) gives the first part of \cite[Corollary~1.3]{Z}.

\section{Correspondence of closed elements}

In this section we establish a bijective correspondence between the sets of closed elements induced by some special 
Galois connections. Our results not only dualize, but extend those from \cite{CIKO}. 

\begin{thm} \label{t:main} Let $(\alpha,\beta)$ be an essential retractable UC Galois connection between two bounded
lattices $A$ and $B$. For $x\in A\cup B$ we denote by $\bar{x}$ the unique closure of $x$, when it does exist. Denote
\begin{align*} \mathcal{A}&=\{a\in A\mid a \textrm{ is closed in $A$ and $\alpha(a)$ has a unique closure in $B$}\}, \\ 
\mathcal{B}&=\{b\in B\mid b \textrm{ is closed in $B$ and $\beta(b)$ has a unique closure in $A$}\}.
\end{align*}
Consider $\varphi:\mathcal{A}\to \mathcal{B}$ defined by $\varphi(a)=\overline{\alpha(a)}$ for every $a\in \mathcal{A}$,
and $\psi:\mathcal{B}\to \mathcal{A}$ defined by $\psi(b)=\overline{\beta(b)}$ for every $b\in \mathcal{B}$.
Then $\varphi$ is a well-defined map. Also, $\psi$ is a well-defined map if and only if $\mathcal{B}$ coincides with the
set of closed elements $b\in B$ such that $\beta(b)$ is closed in $B$. In this case, $\varphi$ and $\psi$ are mutually
inverse bijections. 
\end{thm}

\begin{proof} This is dual to \cite[Theorem~3.5]{CIKO}, using Lemmas \ref{l:ess}, \ref{l:clgal} and \ref{l:transfer}.
\end{proof}

For the next theorem, which refines Theorem \ref{t:main} and generalizes the dual of \cite[Theorem~3.6]{CIKO}, we need
the modularity of the lattices involved. We include its proof for the reader's convenience, and for pointing out easier
how to dualize it later on.

\begin{thm} \label{t:main2} Let $(\alpha,\beta)$ be an essential retractable UC Galois connection between two bounded
modular lattices $A$ and $B$. Then there are mutually inverse bijections between the sets $\mathcal{C}_A$ of closed
elements in $A$ and $\mathcal{C}_B$ of closed elements in $B$. If $B$ is UC, then these bijections are order-preserving.
\end{thm}

\begin{proof} We use the notation from Theorem \ref{t:main}. Hence $\mathcal{A}$ is the set of all closed elements $a\in
A$ such that $\alpha(a)$ has a unique closure in $B$, $\mathcal{B}$ is the set of all closed elements $b\in B$ such that
$\beta(b)$ is closed in $A$, and the maps $\varphi:\mathcal{A}\to \mathcal{B}$ given by
$\varphi(a)=\overline{\alpha(a)}$ for every $a\in \mathcal{A}$, and $\psi:\mathcal{B}\to \mathcal{A}$ given by
$\psi(b)=\beta(b)$ for every $b\in \mathcal{B}$, are well-defined mutually inverse bijections.
Now it is enough to prove that $\mathcal{A}=\mathcal{C}_A$ and $\mathcal{B}=\mathcal{C}_B$. Clearly,
$\mathcal{A}\subseteq \mathcal{C}_A$ and $\mathcal{B}\subseteq \mathcal{C}_B$.  

Let $a\in \mathcal{C}_A$. We claim that $\alpha(a)$ has a unique closure in $B$. The existence of a closure of
$\alpha(a)$ in $B$ follows by Lemma \ref{l:ess}. Now assume that $b_1$ and $b_2$ are two coclosures of
$\alpha(a)$ in $B$. Then $\alpha(a)$ is closed both in $[0,b_1]$ and in $[0,b_2]$. By Lemma \ref{l:transfer},
$\beta\alpha(a)$ is essential both in $[0,\beta(b_1)]$ and in $[0,\beta(b_2)]$. Since $a$ is closed, we have
$\beta\alpha(a)=a$ by Lemma
\ref{l:clgal}. Hence $a$ is essential both in $[0,\beta(b_1)]$ and in $[0,\beta(b_2)]$. But $a$ is closed in $A$,
hence we must have $\beta(b_1)=\beta(b_2)=a$. Since $(\alpha,\beta)$ is retractable, $\alpha\beta(b_1)=\alpha\beta(b_2)$
is essential in $[0,b_1]$. Then $b_1$ is a closure of $\alpha\beta(b_2)$ in $B$.
Since $(\alpha,\beta)$ is UC, $b_2$ is the unique closure of $\alpha\beta(b_2)$ in $B$, hence $b_1=b_2$. Then
$\alpha(a)$ has a unique closure in $B$. Thus $a\in \mathcal{A}$. 

Now let $b\in \mathcal{C}_B$. We claim that $\beta(b)$ is closed in $A$. To this end, let $a'\in A$ be such that
$\beta(b)$ is essential in $[0,a']$. By Lemma \ref{l:transfer}, $\alpha\beta(b)$ is essential in $[0,\alpha(a')]$.
Since $B$ is modular, Lemma \ref{l:ess} yields a closure of $\alpha(a')$ in $B$, say $b_0$. Then
$\alpha(a')$ is essential in $[0,b_0]$, whence $\alpha\beta(a)$ is essential in $[0,b_0]$ by Lemma \ref{l:ess}.
Hence $b_0$ is a closure of $\alpha\beta(b)$ in $B$. Since $(\alpha,\beta)$ is UC, $b$ is the unique
closure of $\alpha\beta(b)$ in $B$, hence $b_0=b$. Then $\alpha(a')\leq b$, which implies $a'\leq \beta(b)$ because
$(\alpha,\beta)$ is a Galois connection. Hence $\beta(b)=a'$, and so $\beta(b)$ is closed in $A$. Thus $b\in
\mathcal{B}$.

It follows that $\mathcal{A}=\mathcal{C}_A$ and $\mathcal{B}=\mathcal{C}_B$, as desired. Note that $\psi$ is
order-preserving, and by Lemma \ref{l:ess}, $\varphi$ is order-preserving if $B$ is UC.
\end{proof}

\begin{cor} \label{c:bij} Let $(\alpha,\beta)$ be an essential retractable UC Galois connection between two bounded
modular lattices $A$ and $B$. Then the following are equivalent:

(i) The mutually inverse bijections from Theorem \ref{t:main2} are the restrictions of $\alpha$ and $\beta$ to
the sets $\mathcal{C}_A$ of coclosed elements in $A$ and $\mathcal{C}_B$ of coclosed elements in $B$ respectively.

(ii) Every closed element in $B$ is Galois.
\end{cor}

\begin{proof} This is immediate by Theorems \ref{t:main} and \ref{t:main2}.
\end{proof}

\begin{rem} \rm Under the hypotheses of Theorem \ref{t:main2}, every closed element in $A$ is Galois by
Lemma \ref{l:clgal}. In case every closed element in $B$ is Galois, Theorem \ref{t:main2} establishes in fact a
bijection between the sets of closed Galois elements in $A$ and $B$. We point out that this is not true in general.
Indeed, let $(\alpha,\beta)$ be the essential retractable UC Galois connection from Example \ref{e:ex} (1). The closed
Galois elements in the domain $A=L(G)$ of $\alpha$ are $0$, $H_3$, $H_5$ and $G$, whereas the only closed Galois
element in the domain $B=L(G)$ of $\beta$ is $0$. Hence there is no bijection between the two sets of
closed Galois elements. Note that there are closed elements in $B=L(G)$ which are not Galois, for instance $H_3$.
\end{rem}

The following module-theoretic corollary follows immediately by Theorems \ref{t:retcoret} and \ref{t:main2}.

\begin{cor} {\rm \cite[Theorem~1.2]{Z}} Assume that $N$ is $M$-faithful. Then there are mutually inverse bijections
between the set of closed submodules of $\Hom_R(M,N)_S$ and the set of closed submodules of $N_R$. If $N_R$ is UC, then
these bijections preserve inclusions. 
\end{cor}

Let us recall the following notion, which is the lattice-theoretic version of the corresponding one for modules
(e.g., see \cite{DHSW}).

\begin{defn} \label{d:extending} \rm Let $(A,\wedge,\vee,0,1)$ be a bounded modular lattice. Then $A$ is called {\it
extending} if for every $a\in A$ there exists a complement $a'\in A$ such that $a$ is essential in $[0,a']$.
Equivalently, $A$ is extending if and only if every closed element $a\in A$ is a complement.
\end{defn}

Now we may relate extending-type properties of some bounded modular lattices. 

\begin{thm} \label{t:extending} Let $(\alpha,\beta)$ be an essential retractable UC Galois connection between two
bounded modular lattices $A$ and $B$.

(i) Assume that $\beta$ is additive. If $B$ is extending, then so is $A$. 

(ii) Assume that $\alpha(1)=1$. If $A$ is extending, then so is $B$. 
\end{thm}

\begin{proof} (i) Assume that $B$ is extending. Let $a$ be a closed element in $A$. Then
$\varphi(a)=\overline{\alpha(a)}$ is closed in $B$ by Theorem \ref{t:main2}. By hypotheses, $\overline{\alpha(a)}$ is a
complement in $B$, hence $a=\beta(\overline{\alpha(a)})$ is a complement in $A$, because $\beta$ is additive. Thus $A$
is extending.

(ii) Assume that $A$ is extending. Let $b$ be a closed element in $B$. Then $\psi(b)=\beta(b)$ is closed in $A$ by
Theorem \ref{t:main2}. Since $A$ is extending, $\beta(b)$ is a complement in $A$. Then $\alpha\beta(b)$ is a complement
in $B$ by Lemma \ref{l:clgal}. Hence there there exists $b'\in B$ such that $\alpha\beta(b)\vee b'=1$ and
$\alpha\beta(b)\wedge b'=0$. Then clearly $b\vee b'=1$, and by the retractability of $(\alpha,\beta)$, $b\wedge b'=0$.
Hence $b$ is a complement in $B$, and so $B$ is extending.
\end{proof}

Concerning the module-theoretic Galois connection $(r_N',l_U')$, $l_U'$ is clearly additive, and we have 
$r_N'(\Hom_R(M,N))=N$ if and only if $\Hom_R(M,N)M=N$ if and only if $N$ is $M$-generated. Now Theorems \ref{t:retcoret}
and \ref{t:extending} yield the following consequence.

\begin{cor} \cite[Corollary~1.3]{Z} Assume that $N$ is $M$-faithful. If $N_R$ is extending, then so is $\Hom_R(M,N)_S$.
The converse holds if $N$ is $M$-generated.  
\end{cor}

\section{Dualizations}

By considering the concepts of essential and closed element, closure of an element, UC, uniform and extending lattice in
the set-dual $A^{\rm op}$ of a bounded lattice $(A,\wedge,\vee,0,1)$, one obtains the notions of {\it coessential} (or
{\it cosmall}) and {\it coclosed} element, {\it coclosure} of an element, {\it UCC}, {\it hollow} and {\it
lifting} lattice respectively (see \cite{DHSW} for the module-theoretic notions). The results concerning all these
notions are in general dualizable, except for those involving the existence of a (co)closure of an element. While every
element in a bounded modular lattice has a closure, this is not the case for coclosures. For instance, the subgroup
$2\mathbb{Z}$ of the abelian group $\mathbb{Z}$ has no coclosure in the subgroup lattice of $\mathbb{Z}$
(\cite[3.10]{CLVW}). When the bounded modular lattice $A$ is \emph{amply supplemented}, in the sense that every element
$a\in A$ has a supplement $x\in A$ such that $a$ is coessential in $[x,1]$, then every element of $A$ has a coclosure in
$A$ (e.g., see \cite[Lemma~3.3]{CIKO}). 

Let us point out that $(\alpha,\beta)$ is a Galois connection between two bounded lattices $A$ and $B$ if and only if
$(\beta,\alpha)$ is a Galois connection between $B^{\rm op}$ and $A^{\rm op}$. The concepts of essential,
retractable and UC Galois connection are dualized as follows. 

\begin{defn} \rm A Galois connection $(\alpha,\beta)$ between two bounded lattices $A$ and $B$ is called: 

(1) {\it coessential} if for every $b\in B$, $b$ is coessential in $[\alpha\beta(b),1]$.

(2) {\it coretractable} if for every $a\in A$, $\beta\alpha(a)$ is coessential in $[a,1]$.

(3) {\it UCC} if for every coclosed element $a\in A$, $a$ is the unique coclosure of $\beta\alpha(a)$ in $A$. 
\end{defn}

\begin{rem} \rm (1) Note that a coretractable Galois connection was called cosmall in \cite[Definition~2.6]{CIKO}. We
prefer to use the present name due to the fact that it is a dualization of a retractable Galois connection, which in
turn corresponds to the already developed concept of retractable modules. 

(2) Let $(\alpha,\beta)$ be a Galois connection between two bounded lattices $A$ and $B$. Since $\alpha\beta(1)=1$,
the condition that $\beta$ preserves finite suprema, used in \cite{CIKO}, implies that $(\alpha,\beta)$ is coessential
(see \cite[Lemma~2.5]{CIKO}). This also ensures that $\alpha(1)=1$, dually to Lemma \ref{l:clgal} (i).  
\end{rem}

We leave to the reader the statements of the duals of the results from the previous sections and their applications to
particular (module-theoretic) Galois connections, such as $(r_M,l_U)$ or $(r'_N,l'_U)$. Nevertheless, we state the dual
of Theorem \ref{t:main2}, which involves the extra condition of an amply supplemented latice. 

\begin{thm} \label{t:hmain2} Let $(\alpha,\beta)$ be a coessential coretractable UCC Galois connection between two
bounded modular lattices $A$ and $B$ such that $A$ is amply supplemented. Then there are mutually inverse bijections
between the set of coclosed elements in $A$ and the set of coclosed elements in $B$. If $A$ is UCC, then these
bijections are order-preserving.
\end{thm}


\begin{thebibliography}{99}

\bibitem{AN} T. Albu and C. N\u ast\u asescu, \emph{Relative finiteness in module theory}, Marcel Dekker, 1984.

\bibitem{AES} B. Amini, M. Ershad and H. Sharif, \emph{Coretractable modules}, J. Aust. Math. Soc. {\bf 86} (2009),
289--304. 

\bibitem{CLVW} J. Clark, C. Lomp, N. Vanaja and R. Wisbauer, \emph{Lifting modules}, Frontiers in Mathematics,
Birkh\"auser, 2006.

\bibitem{CIKO} S. Crivei, H. Inank\i l, M.T. Ko\c san and G. Olteanu, \emph{Correspondences of coclosed submodules},
Comm. Algebra, 2013, to appear. arXiv:1203.0729

\bibitem{CO} S. Crivei and G. Olteanu, {\it GAP algorithms for finite abelian groups and applications}, 
Carpathian J. Math. {\bf 24} (2008), 310--316.

\bibitem{CS} S. Crivei and \c S. \c Suteu Sz\"oll\H{o}si, \emph{Subgroup lattice algorithms related to extending and
lifting abelian groups}, Int. Electron. J. Algebra {\bf 2} (2007), 54--70.

\bibitem{DHSW} N.V. Dung, D.V.  Huynh, P.F.  Smith and R. Wisbauer, \emph{Extending modules}, Pitman Research Notes,
{\bf 313}, Longman Scientific and Technical, 1994.

\bibitem{Erne} M. Ern\'e, J. Koslowski, A. Melton and G.E. Strecker, \emph{A primer on Galois connections}. In:
Proceedings of the 1991 Summer Conference on General Topology and Applications in Honor of Mary Ellen Rudin and Her
Work, Annals of the New York Academy of Sciences, {\bf 704}, 1993, pp. 103--125.

\bibitem{GP} P. Grzeszczuk and E.R. Puczy\l owski, \emph{On Goldie and dual Goldie dimension}, J. Pure Appl. Algebra
{\bf 31} (1984), 47--54.

\bibitem{HV} A. Haghany and M.R. Vedadi, \emph{Study of semi-projective retractable modules}, Algebra Colloq. {\bf 14}
(2007), 489--496.

\bibitem{Schmidt} R. Schmidt, \emph{Subgroup lattices of groups}, Walter de Gruyter, 1994. 

\bibitem{Wis} R. Wisbauer, \emph{Foundations of ring and module theory}, Gordon and Breach, 1991.

\bibitem{Z} J.M. Zelmanowitz, \emph{Correspondences of closed submodules}, Proc. Amer. Math. Soc. {\bf 124} (1996),
2955-- 2960.



\end{thebibliography}
\end{document}